\documentclass[11pt]{amsart}
\usepackage[margin=1.15in]{geometry}
\usepackage{amsmath,amssymb}
\usepackage{booktabs}
\usepackage{graphicx}
\usepackage[hidelinks]{hyperref}
\usepackage{url}

\newtheorem{theorem}{Theorem}
\newtheorem{lemma}[theorem]{Lemma}
\newtheorem{proposition}[theorem]{Proposition}
\newtheorem{remark}[theorem]{Remark}
\newcommand{\Kq}[3]{K_{#1}(#2,#3)}
\newcommand{\Zq}{\mathbb{Z}_q}

\newcommand{\ubkeylegend}{\texttt{c} = monotonicity $K_q(n{+}1,R{+}1)\le K_q(n,R)$; \texttt{d} = Bhandari--Durairajan~\cite{BhandariDurairajan1996}; \texttt{e} = $K_q(n{+}1,R)\le q\,K_q(n,R)$; \texttt{f} = direct sum; \texttt{j} = alphabet product rule; \texttt{n} = K\'eri--\"Osterg\aa rd~\cite{KeriOstergard2005}; \texttt{q} = Rivas Soriano (attributed in~\cite{Keri}); \texttt{z} = K\'eri 2009 updates recorded in~\cite{Keri}}
\newcommand{\ubtable}{%
\begin{tabular}{lccccccc}
\toprule
cell & previous bounds & sphere & new UB & $\Delta$ & $\Delta\%$ & key & via \\
\midrule
$\Kq{5}{11}{5}$ & 103--625 & 86 & \textbf{602} & $-23$ & $3.7\%$ & \texttt{d} & L \\
$\Kq{6}{7}{3}$ & 70--246 & 57 & \textbf{227} & $-19$ & $7.7\%$ & \texttt{e} & L \\
$\Kq{6}{8}{3}$ & 246--1080 & 217 & \textbf{1005} & $-75$ & $6.9\%$ & \texttt{f} & L \\
$\Kq{6}{8}{4}$ & 46--216 & 33 & \textbf{166} & $-50$ & $23.1\%$ & \texttt{e} & L \\
$\Kq{6}{9}{4}$ & 136--738 & 112 & \textbf{660} & $-78$ & $10.6\%$ & \texttt{f} & L \\
$\Kq{6}{9}{5}$ & 32--144 & 21 & \textbf{119} & $-25$ & $17.4\%$ & \texttt{j} & L \\
$\Kq{6}{10}{4}$ & \textbf{441}--2952 & 411 & \textbf{2751} & $-201$ & $6.8\%$ & \texttt{f} & L \\
$\Kq{6}{10}{5}$ & 83--615 & 65 & \textbf{484} & $-131$ & $21.3\%$ & \texttt{f} & L \\
$\Kq{7}{8}{4}$ & 76--343 & 56 & \textbf{316} & $-27$ & $7.9\%$ & \texttt{c} & L \\
$\Kq{7}{9}{4}$ & 264--1843 & 221 & \textbf{1475} & $-368$ & $20.0\%$ & \texttt{f} & L \\
$\Kq{7}{9}{5}$ & 52--323 & 35 & \textbf{240} & $-83$ & $25.7\%$ & \texttt{f} & L \\
$\Kq{7}{10}{5}$ & 160--1225 & 126 & \textbf{980} & $-245$ & $20.0\%$ & \texttt{f} & T \\
$\Kq{8}{6}{4}$ & 15--20 & 7 & \textbf{19} & $-1$ & $5.0\%$ & \texttt{n} & L \\
$\Kq{8}{8}{4}$ & 118--512 & 89 & \textbf{505} & $-7$ & $1.4\%$ & \texttt{c} & L \\
$\Kq{8}{8}{5}$ & 29--90 & 15 & \textbf{79} & $-11$ & $12.2\%$ & \texttt{z} & L \\
$\Kq{8}{10}{5}$ & 287--2461 & 225 & \textbf{1883} & $-578$ & $23.5\%$ & \texttt{f} & T \\
$\Kq{8}{10}{6}$ & 58--342 & 37 & \textbf{324}$^{\downarrow}$ & $-18$ & $5.3\%$ & \texttt{n} & L \\
$\Kq{9}{9}{5}$ & 120--729 & 83 & \textbf{684} & $-45$ & $6.2\%$ & \texttt{c} & T \\
$\Kq{9}{10}{5}$ & 481--3969 & 380 & \textbf{3393} & $-576$ & $14.5\%$ & \texttt{f} & T \\
$\Kq{10}{9}{5}$ & 174--1088 & 121 & \textbf{1055} & $-33$ & $3.0\%$ & \texttt{f} & T \\
$\Kq{10}{10}{5}$ & 632--7106 & 612 & \textbf{5799} & $-1307$ & $18.4\%$ & \texttt{f} & T \\
$\Kq{10}{10}{6}$ & 122--826 & 79 & \textbf{802}$^{\downarrow}$ & $-24$ & $2.9\%$ & \texttt{n} & T \\
$\Kq{12}{6}{4}$ & 30--41 & 13 & \textbf{39} & $-2$ & $4.9\%$ & \texttt{q} & L \\
$\Kq{13}{6}{4}$ & 35--46 & 14 & \textbf{45} & $-1$ & $2.2\%$ & \texttt{q} & L \\
$\Kq{14}{6}{4}$ & 40--52 & 16 & \textbf{50} & $-2$ & $3.8\%$ & \texttt{q} & L \\
$\Kq{15}{6}{4}$ & 46--59 & 18 & \textbf{57} & $-2$ & $3.4\%$ & \texttt{c} & L \\
\bottomrule
\end{tabular}}
\newcommand{\lbtable}{%
\begin{tabular}{lccc}
\toprule
cell & sphere & prev.\ LB & new LB \\
\midrule
$\Kq{6}{8}{2}$ & 2267 & 2276 (\texttt{y}) & \textbf{2367} \\
$\Kq{6}{9}{2}$ & 10653 & 10900 (\texttt{y}) & \textbf{10965} \\
$\Kq{6}{9}{3}$ & 881 & 921 (\texttt{y}) & \textbf{926} \\
$\Kq{6}{10}{3}$ & 3739 & 3815 (\texttt{y}) & \textbf{3836} \\
$\Kq{6}{10}{4}$ & 411 & 417 (\texttt{s}) & \textbf{441} \\
$\Kq{7}{7}{2}$ & 1031 & 1035 (\texttt{x}) & \textbf{1081} \\
$\Kq{7}{8}{2}$ & 5454 & 5457 (\texttt{x}) & \textbf{5631} \\
$\Kq{7}{8}{3}$ & 439 & 457 (\texttt{y}) & \textbf{471} \\
$\Kq{7}{9}{2}$ & 29870 & 29889 (\texttt{s}) & \textbf{30562} \\
$\Kq{7}{9}{3}$ & 2070 & 2077 (\texttt{y}) & \textbf{2143} \\
$\Kq{7}{10}{2}$ & 168041 & 168042 (\texttt{x}) & \textbf{170632} \\
$\Kq{8}{8}{2}$ & 11741 & 11766 (\texttt{y}) & \textbf{12033} \\
$\Kq{8}{8}{3}$ & 813 & 829 (\texttt{y}) & \textbf{856} \\
$\Kq{8}{9}{4}$ & 403 & 409 (\texttt{s}) & \textbf{429} \\
$\Kq{9}{8}{2}$ & 23181 & 23184 (\texttt{x}) & \textbf{23642} \\
$\Kq{9}{8}{3}$ & 1411 & 1413 (\texttt{x}) & \textbf{1464} \\
$\Kq{9}{9}{3}$ & 8537 & 8544 (\texttt{y}) & \textbf{8685} \\
$\Kq{9}{9}{4}$ & 690 & 703 (\texttt{s}) & \textbf{722} \\
$\Kq{9}{10}{3}$ & 54142 & 54144 (\texttt{x}) & \textbf{54600} \\
$\Kq{10}{7}{2}$ & 5666 & 5676 (\texttt{y}) & \textbf{5824} \\
$\Kq{10}{8}{2}$ & 42717 & 42772 (\texttt{y}) & \textbf{43423} \\
$\Kq{10}{9}{2}$ & 333556 & 333560 (\texttt{x}) & \textbf{337394} \\
$\Kq{10}{9}{4}$ & 1123 & 1130 (\texttt{s}) & \textbf{1160} \\
$\Kq{10}{10}{2}$ & 2676660 & 2676660 (\texttt{a}) & \textbf{2699348} \\
$\Kq{10}{10}{4}$ & 6808 & 6886 (\texttt{y}) & \textbf{6915} \\
$\Kq{11}{7}{2}$ & 8977 & 9106 (\texttt{y}) & \textbf{9193} \\
$\Kq{11}{8}{2}$ & 74405 & 74415 (\texttt{s}) & \textbf{75448} \\
$\Kq{12}{8}{2}$ & 123665 & 123772 (\texttt{y}) & \textbf{125156} \\
$\Kq{12}{8}{3}$ & 5512 & 5577 (\texttt{y}) & \textbf{5605} \\
\bottomrule
\end{tabular}\hspace{1.5em}\begin{tabular}{lccc}
\toprule
cell & sphere & prev.\ LB & new LB \\
\midrule
$\Kq{13}{6}{2}$ & 2162 & 2169 (\texttt{x}) & \textbf{2233} \\
$\Kq{13}{7}{2}$ & 20183 & 20189 (\texttt{x}) & \textbf{20545} \\
$\Kq{13}{8}{2}$ & 197562 & 197563 (\texttt{x}) & \textbf{199633} \\
$\Kq{13}{8}{3}$ & 8085 & 8086 (\texttt{x}) & \textbf{8193} \\
$\Kq{14}{6}{2}$ & 2881 & 2955 (\texttt{y}) & \textbf{2964} \\
$\Kq{14}{7}{2}$ & 28952 & 29273 (\texttt{y}) & \textbf{29404} \\
$\Kq{14}{8}{2}$ & 305105 & 305294 (\texttt{y}) & \textbf{307909} \\
$\Kq{15}{6}{2}$ & 3766 & 3812 (\texttt{y}) & \textbf{3862} \\
$\Kq{15}{8}{2}$ & 457578 & 457584 (\texttt{x}) & \textbf{461294} \\
$\Kq{16}{6}{2}$ & 4841 & 4848 (\texttt{x}) & \textbf{4951} \\
$\Kq{16}{7}{2}$ & 55566 & 55600 (\texttt{y}) & \textbf{56237} \\
$\Kq{16}{8}{2}$ & 668894 & 669207 (\texttt{y}) & \textbf{673723} \\
$\Kq{17}{7}{2}$ & 74757 & 75429 (\texttt{y}) & \textbf{75559} \\
$\Kq{17}{8}{2}$ & 955977 & 955978 (\texttt{x}) & \textbf{962145} \\
$\Kq{17}{8}{3}$ & 29475 & 29478 (\texttt{x}) & \textbf{29652} \\
$\Kq{17}{8}{4}$ & 1446 & 1458 (\texttt{y}) & \textbf{1464} \\
$\Kq{18}{6}{2}$ & 7664 & 7741 (\texttt{y}) & \textbf{7804} \\
$\Kq{18}{8}{2}$ & 1339162 & 1339650 (\texttt{y}) & \textbf{1346931} \\
$\Kq{19}{6}{2}$ & 9468 & 9479 (\texttt{x}) & \textbf{9624} \\
$\Kq{19}{7}{2}$ & 128968 & 128972 (\texttt{x}) & \textbf{130081} \\
$\Kq{19}{7}{3}$ & 4236 & 4237 (\texttt{x}) & \textbf{4282} \\
$\Kq{19}{8}{2}$ & 1842635 & 1842639 (\texttt{x}) & \textbf{1852296} \\
$\Kq{20}{7}{2}$ & 165911 & 167165 (\texttt{y}) & \textbf{167204} \\
$\Kq{20}{7}{3}$ & 5166 & 5174 (\texttt{x}) & \textbf{5215} \\
$\Kq{20}{8}{2}$ & 2494884 & 2495614 (\texttt{y}) & \textbf{2506759} \\
$\Kq{21}{6}{2}$ & 14012 & 14131 (\texttt{y}) & \textbf{14201} \\
$\Kq{21}{7}{3}$ & 6243 & 6249 (\texttt{x}) & \textbf{6294} \\
$\Kq{21}{8}{2}$ & 3329184 & 3329193 (\texttt{x}) & \textbf{3343629} \\
$\Kq{21}{8}{3}$ & 82338 & 82341 (\texttt{x}) & \textbf{82595} \\
\bottomrule
\end{tabular}}

\title[New bounds on covering codes for $5\le q\le 21$]{New upper and lower
bounds on covering codes $K_q(n,R)$\\ for alphabets of size $5\le q\le 21$}
\author{M\'ark Marosi}
\thanks{Department of Artificial Intelligence and Systems Engineering,
Budapest University of Technology and Economics (BME), Budapest, Hungary.
Email: \texttt{marosi@mit.bme.hu}.}
\date{September 2026}

\begin{document}

\begin{abstract}
Let $K_q(n,R)$ denote the minimum cardinality of a $q$-ary code of length
$n$ with covering radius $R$. We improve the known bounds on $K_q(n,R)$ in
$84$ cases ($83$ distinct cells). Twenty-six upper bounds for $5\le q\le
15$ are established by explicit codes, found by local search and by a large
neighbourhood search whose evaluations are exact coverage counts over the
whole space $\mathbb{Z}_q^n$, computed by a coordinate-wise transform; two
examples are $\Kq{7}{9}{5}\le 240$ (previously $323$) and
$\Kq{8}{10}{5}\le 1883$ (previously $2461$). Fifty-eight lower bounds for
$6\le q\le 21$ are obtained from the semidefinite programming bound of
Gijswijt and Polak, whose published computations cover $q\le 5$, by solving
the symmetry-reduced program in multiprecision arithmetic and rounding each
dual solution to an exact rational certificate; every certificate is
checked by a standalone program in exact arithmetic. The cell
$\Kq{6}{10}{4}$ is improved from both sides, from $417$--$2952$ to
$441$--$2751$. The codes, the certificates, and the checkers are provided
as ancillary files.
\end{abstract}

\maketitle

\section{Introduction}

Let $\Zq^n$ denote the Hamming space of $q$-ary words of length $n$. A code
$C\subseteq \Zq^n$ has \emph{covering radius} at most $R$ if every word of
$\Zq^n$ is within Hamming distance $R$ of some codeword, and
\[
K_q(n,R) \;=\; \min\{\,|C| : C\subseteq \Zq^n \text{ has covering radius}
\le R\,\}.
\]
The binary case $K_2(n,1)$ is the domination number of the hypercube, and
$K_3(n,1)$ is the football pool problem; see~\cite{CHLL} for the general
theory.

Upper bounds on $K_q(n,R)$ are established by explicit codes and lower
bounds by counting or convex-relaxation arguments. The known bounds for
$q\ge 3$ are collected in the tables of K\'eri~\cite{Keri}, which extend
those of~\cite{CHLL} and were last revised in November 2011; the tables
attribute $96$ of their upper bounds for $q\ge 6$
to~\cite{KeriOstergard2005}. Recent lower-bound work concerns small
alphabets: Wu and Chen~\cite{WuChen} improved binary lower bounds, and
Gijswijt and Polak~\cite{GijswijtPolak} obtained new lower bounds for
$q\le 5$ from a semidefinite programming hierarchy. Florath~\cite{Florath}
formalizes covering-code bounds in a proof assistant. For $q\ge 6$ we know
of no published improvement on either side since 2011, and of none on the
upper-bound side for $q\ge 5$; a search of arXiv, DBLP, OpenAlex,
Lobstein's covering-radius bibliography~\cite{Lobstein}, and the
publication lists of the authors active in the area, carried out in August
2026, found none.

The contributions of this paper are as follows.
\begin{enumerate}
\item Twenty-six new upper bounds for $5\le q\le 15$
(Table~\ref{tab:ub}), each established by an explicit code. The codes were
found by a local search (Section~\ref{sec:ls}), effective for $q^n\lesssim
10^8$, and by a large neighbourhood search (Section~\ref{sec:lns}) whose
evaluations are exact coverage counts over all of $\Zq^n$, which extends
the searchable range to $q^n=10^{10}$.
\item Fifty-eight new lower bounds for $6\le q\le 21$
(Table~\ref{tab:lb}), obtained by solving the symmetry-reduced semidefinite
program of~\cite{GijswijtPolak} in multiprecision arithmetic and rounding
each numerical dual solution to an exact rational certificate
(Section~\ref{sec:lb}). Each certificate is checked by a standalone program
that rebuilds the program from $(q,n,R)$ and verifies the certificate in
exact arithmetic.
\item Every code was verified by four methods, and every certificate by
the exact checker, from the published files
(Section~\ref{sec:verify}). The codes, the certificates, and both checkers
are ancillary files of this paper.
\end{enumerate}

Section~\ref{sec:prelim} fixes notation, Section~\ref{sec:results} states
the results, Sections~\ref{sec:ls} and~\ref{sec:lns} describe the two
upper-bound searches, Section~\ref{sec:lb} the lower-bound computation, and
Section~\ref{sec:verify} the verification. Implementation measurements are
collected in Appendix~\ref{app:impl}, the certified program is stated in
Appendix~\ref{app:sdp}, and the code establishing $\Kq{6}{8}{4}\le 166$ is
printed in Appendix~\ref{app:code}.

\section{Preliminaries}\label{sec:prelim}

Throughout, $q\ge 2$ and $\Zq=\{0,\dots,q-1\}$; $d(x,y)$ denotes Hamming
distance on $\Zq^n$, $B_R(x)=\{y: d(x,y)\le R\}$ the Hamming ball, and
$|B_R|=\sum_{i\le R}\binom{n}{i}(q-1)^i$ its volume. Counting gives the
\emph{sphere-covering bound}
$K_q(n,R)\ \ge\ \lceil q^n/|B_R|\rceil$.

The constructions that propagate upper bounds through the tables
of~\cite{Keri} are the \emph{direct sum}
\begin{equation}\label{eq:directsum}
K_q(n_1+n_2,\,R_1+R_2)\;\le\;K_q(n_1,R_1)\cdot K_q(n_2,R_2),
\end{equation}
the inequalities $K_q(n+1,R)\le q\,K_q(n,R)$ and $K_q(n+1,R+1)\le
K_q(n,R)$, and product rules over factorizations of the alphabet;
see~\cite{CHLL,Keri} for these and for the source annotations (``keys'')
of~\cite{Keri} reproduced in the tables below. For a code $C$ we write
$\mathrm{cnt}(w)=\#\{c\in C: d(w,c)\le R\}$ for the number of codewords
covering a word $w$; $C$ has covering radius at most $R$ if and only if
$\mathrm{cnt}(w)>0$ for every $w$.

\section{The new bounds}\label{sec:results}

Tables~\ref{tab:ub} and~\ref{tab:lb} list the new bounds and
Figure~\ref{fig:summary} shows both sides. The cell $\Kq{6}{10}{4}$ is
improved from both sides, from $417$--$2952$ to $441$--$2751$, and is the
only cell appearing in both tables.

\emph{Upper bounds.} For $24$ of the $26$ bounds of Table~\ref{tab:ub}, a
search for a code of size $M-1$ with the final computational budget found
none. For the two bounds marked ${}^{\downarrow}$ the search was still
producing smaller codes when this version was prepared. The largest
relative improvements occur on cells whose previous bound was inherited
from a general construction (keys \texttt{c}, \texttt{d}, \texttt{e},
\texttt{f}, \texttt{j}); the cells whose previous bound was obtained by
search in~\cite{KeriOstergard2005} or recorded in~\cite{Keri} (keys
\texttt{n}, \texttt{q}, \texttt{z}) improve by between $2.2\%$ and
$12.2\%$. The two searches interact through the tables: the new
$\Kq{7}{8}{4}\le 316$ gives $\Kq{7}{9}{5}\le 316$ by monotonicity, and the
search in that cell, started from the propagated code, reached
$\Kq{7}{9}{5}\le 240$. The closure of all new bounds under the rules of
Section~\ref{sec:prelim} gives no further improvement beyond the
observation in Remark~\ref{rem:keri17}.

\emph{Lower bounds.} All $58$ improved cells carry K\'eri keys \texttt{y},
\texttt{x}, \texttt{s} or \texttt{a}, that is, their previous lower bound
was the sphere-covering bound or a counting refinement of it. The entire
family $(n,R)=(8,2)$, $6\le q\le 21$, is improved, by increments from
$+91$ at $q=6$ to $+14436$ at $q=21$; the largest single increment is
$\Kq{10}{10}{2}\ge 2699348$, previously $2676660$. No cell whose previous
lower bound is due to Haas, Halupczok and Schlage-Puchta~\cite{HHSP2009}
(key \texttt{m}) is improved; Section~\ref{sec:lb:limits} reports the
values obtained on such cells.

\begin{table}[t]
\centering
\small
\ubtable
\caption{New upper bounds. ``Previous bounds'' are the lower and upper
bounds of~\cite{Keri}; a bold lower bound is from Table~\ref{tab:lb}.
``Sphere'' is $\lceil q^n/|B_R|\rceil$. ``Key'' is the source annotation
of the previous upper bound in~\cite{Keri}: \ubkeylegend. ``Via'' is the
method that found the code: L $=$ the local search of
Section~\ref{sec:ls}, T $=$ the large neighbourhood search of
Section~\ref{sec:lns}. Percentages are rounded to one decimal. For the two
bounds marked ${}^{\downarrow}$ the search was still producing smaller
codes when this version was prepared; for every other bound a search at
$M-1$ with the final budget found no code.}
\label{tab:ub}
\end{table}

\begin{table}[t]
\centering
\footnotesize
\lbtable
\caption{New lower bounds, each certified by an exact rational dual
solution of the program of Appendix~\ref{app:sdp}. ``Sphere'' is
$\lceil q^n/|B_R|\rceil$; ``prev.\ LB'' is the value of~\cite{Keri} with
its key: \texttt{y} = van Wee-type excess counting~\cite{vanWee}
(our reading of the legend of~\cite{Keri}; see also~\cite[Ch.~6]{CHLL}),
\texttt{x} = Chen and Honkala~\cite{ChenHonkala1990}, \texttt{s} = Lang,
Quistorff and Schneider~\cite{LangQuistorffSchneider2007,
LangQuistorffSchneider2008}, \texttt{a} = sphere-covering bound. The exact
rational value of the semidefinite bound (a bound on $K_q(n,R)^3$) is
stored in each certificate.}
\label{tab:lb}
\end{table}

\begin{figure}[t]
\centering
\begin{minipage}[t]{0.48\textwidth}
\centering
\includegraphics[width=\textwidth]{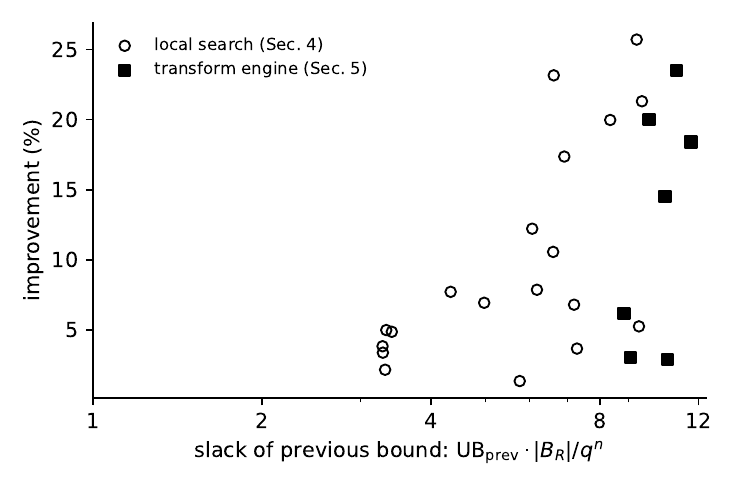}\\[-1mm]
{\small (a)}
\end{minipage}\hfill
\begin{minipage}[t]{0.48\textwidth}
\centering
\includegraphics[width=\textwidth]{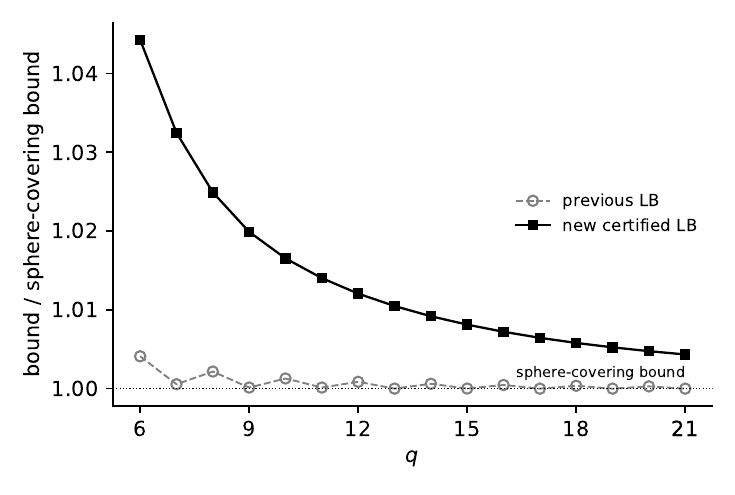}\\[-1mm]
{\small (b)}
\end{minipage}
\caption{(a) Relative improvement of each upper bound of
Table~\ref{tab:ub} against the ratio of the previous bound to the
sphere-covering bound; squares mark the cells found by the large
neighbourhood search of Section~\ref{sec:lns}. (b) The lower bounds for
$(n,R)=(8,2)$, $6\le q\le 21$: previous bound, sphere-covering bound, and
the certified semidefinite bound of Table~\ref{tab:lb}.}
\label{fig:summary}
\end{figure}

\begin{remark}\label{rem:keri17}
The tabulated bound $\Kq{17}{7}{2}\le 252735$ of~\cite{Keri} is weaker
than the bound $17\cdot \Kq{17}{6}{2}\le 17\cdot 14424 = 245208$ that
follows from the tabulated $\Kq{17}{6}{2}\le 14424$ and the inequality
$K_q(n{+}1,R)\le q\,K_q(n,R)$. This involves no new construction.
\end{remark}

\section{Upper bounds by local search}\label{sec:ls}

\subsection{The search}\label{sec:ls:search}

For fixed $(q,n,R,M)$ the search looks for a code of size $M$ minimizing
the number of uncovered words. The state is a multiset of $M$ codewords
together with $\mathrm{cnt}(w)$ for every $w\in\Zq^n$; a move changes one
coordinate of one codeword.

\begin{lemma}\label{lem:move}
Let $c'$ be obtained from $c$ by setting coordinate $p$ to $v\ne c_p$. The
set of words leaving the radius-$R$ ball of $c$ and the set of words
entering the ball of $c'$ are both in bijection with the words at distance
exactly $R$ from $c$ in the coordinates other than $p$. Consequently a move
updates exactly $2S$ counters, $S=\binom{n-1}{R}(q-1)^R$, the two sets are
disjoint, and distinct updates address distinct words.
\end{lemma}

\begin{proof}
Since $d(w,c') = d(w,c) + [\,w_p=c_p\,] - [\,w_p=v\,]$, a word leaves the
ball of $c$ if and only if $d(w,c)=R$ and $w_p=c_p$, and enters the ball
of $c'$ if and only if $d(w,c')=R$ and $w_p=v$. Writing such a word as its
restriction to the coordinates other than $p$ (at distance exactly $R$
from the restriction of $c$) together with its value at $p$ gives the
bijections, disjointness, and injectivity. (For $R\ge n$ both sets are
empty; all cells considered here have $R<n$.)
\end{proof}

The search is a focused local search of WalkSAT type, related to the
simulated annealing and tabu searches used to construct many entries of
the tables~\cite{vanLaarhoven,Ostergard}: select a random uncovered word
$u$; take as candidate moves every codeword at distance exactly $R+1$ from
$u$ moved one coordinate towards $u$ (if no codeword is that close, the
codewords at minimum distance from $u$); evaluate all candidates exactly;
apply a best candidate, ties broken uniformly at random; when no progress
has been made for a fixed number of iterations, move one codeword onto an
uncovered word.

The initial code for a cell is rebuilt from the construction that gave the
tabulated bound (for $\Kq{6}{10}{4}$, the direct sum~\eqref{eq:directsum}
of an optimal $\Kq{6}{4}{1}$ code and a $\Kq{6}{6}{3}$ code of the
tabulated size, both found by search and verified). When a code of size
$M$ is found, the codeword that is the sole coverer of the fewest words is
deleted and the search continues at $M-1$.

\subsection{Evaluation of candidate moves}\label{sec:ls:eval}

Candidate evaluation takes at least $92.9\%$ of the run time on the
reference cells of Appendix~\ref{app:impl:cpu}. The following
reformulation computes the same quantities with less work.

\begin{lemma}\label{lem:shared}
With $D(w)=\{j: w_j\neq c_j\}$, a word $w$ leaves the ball of $c$ under
the move at coordinate $p$ if and only if $|D(w)|=R$ and $p\notin D(w)$;
it enters the ball of the moved codeword if and only if $|D(w)|=R+1$,
$p\in D(w)$ and $w_p=v$.
\end{lemma}

\begin{proof}
$|D(w)|=d(w,c)$, and the move changes the disagreement set only at
coordinate $p$; apply the distance identity in the proof of
Lemma~\ref{lem:move}.
\end{proof}

Let $T=\#\{w: |D(w)|=R,\ \mathrm{cnt}(w)=1\}$ and
$H_j=\#\{w: |D(w)|=R,\ \mathrm{cnt}(w)=1,\ j\in D(w)\}$. By
Lemma~\ref{lem:shared} the number of words uncovered by the move at $p$
equals $T-H_p$, for every $p$ simultaneously: a leaving word with
$\mathrm{cnt}(w)=1$ is covered by $c$ alone, and it is not re-covered by
the moved codeword since $w_p=c_p\ne v$. One enumeration of the
distance-$R$ sphere of $c$ therefore gives the loss of all $n(q-1)$
single-coordinate moves of $c$. The gain of a move depends on the new
value $v$ and concerns only words with $\mathrm{cnt}(w)=0$, so it is
obtained from an explicitly maintained list $U$ of uncovered words in
$O(n|U|)$ operations. Evaluation reads only whether a counter is $0$, $1$
or larger, so a two-bit array $\min(\mathrm{cnt}(w),2)$ serves the reads;
on $\Kq{8}{9}{4}$ this reduces the memory read during evaluation from
$268$\,MB to $33.5$\,MB.

Each reformulation computes the same objective differences as the
reference implementation and therefore follows the same search trajectory
from the same seed; this was checked on $24$ configurations with $2\le
q\le 7$, including $R=0$ and $R=n-1$. The resulting speedup at identical
trajectory is $2$--$10$ times on the reference cells
(Appendix~\ref{app:impl:cpu}). Evaluating all $n(q-1)$ moves of the
candidate codewords instead of the $n$ moves towards $u$ changes the
search; it reduces the time to a fixed target on the two larger reference
cells and increases it on $\Kq{6}{6}{3}$ (Table~\ref{tab:engttt}).

\subsection{Structured codes}\label{sec:ls:structure}

Two restricted searches complement the search over arbitrary subsets of
$\Zq^n$.

\emph{Unions of cosets of a linear code.} When $q$ is a prime power and
$M=c\,q^k$, we search codes of the form $C=\bigcup_{i=1}^{c}(x_i+L)$ with
$L$ a linear $[n,k]_q$ code. Such a $C$ has covering radius at most $R$ if
and only if the syndromes of $x_1,\dots,x_c$ cover the $q^{n-k}$ syndromes
under the coset-leader weights of $L$, a problem on $q^{n-k}$ rather than
$q^n$ elements. This search reproduces the tabulated $\Kq{8}{8}{4}\le
512=8^3$ and $\Kq{5}{11}{5}\le 625=5^4$ in under a second, whereas the
unrestricted search did not reach these sizes; the searches leading to
$505$ and $602$ start from these codes.

\emph{Codes invariant under translation by $\mathbf{1}$.} A second search
works with codes of the form $C=C_0\cup C_1$, where $C_0$ is invariant
under $x\mapsto x+\mathbf{1}$ and is represented in the quotient of size
$q^{n-1}$, and $C_1$ is an arbitrary set of codewords. On $\Kq{6}{8}{4}$,
neither a fully invariant code nor the unrestricted search reached the
target size in our runs, while $C_0\cup C_1$ with a small $C_1$ did so
repeatedly.

\subsection{Configuration}\label{sec:ls:config}

The search program combines four implementations: the search of
Section~\ref{sec:ls:eval}, a SIMD variant of it, and the two restricted
searches of Section~\ref{sec:ls:structure}. The combination was chosen by
comparing the four on a fixed set of instances with held-out evaluation
instances and scoring by verified output only (Appendix~\ref{app:impl:arena});
no single implementation performed best on all instances. A run first
executes the coset search over all feasible $k$ and builds direct sums
over all splits of $n$ and $R$ from recorded codes, then runs a set of
independent search processes whose composition depends on $q^n$ and on
$q^n/(M|B_R|)$; the final code is re-read from disk, duplicates are
removed, and coverage is recomputed by marking balls.

The nineteen codes of Table~\ref{tab:ub} marked L were found by this
program on $64$ CPU cores, with between seconds and a few core-days per
cell; the seven codes marked T are the subject of
Section~\ref{sec:lns}.

\section{Upper bounds by large neighbourhood search with exact coverage
counts}\label{sec:lns}

For $q^n\gtrsim 10^8$ the search of Section~\ref{sec:ls} makes little
progress: the arrays of each process exceed the CPU caches and each move
requires a walk over a sphere of $S$ words at memory latency. (The one
such cell in Table~\ref{tab:ub} found by that search is $\Kq{8}{10}{6}$,
where the radius-$6$ balls are so large that the uncovered set stays
small.) Two ports of the local search to a GPU were implemented and
measured (Appendix~\ref{app:impl:gpu}); their throughput was at most
$1.24$ times that of the $64$-core CPU. The method of this section instead
computes exact coverage counts for every word of $\Zq^n$ at once.

\subsection{Distance-count transforms}

For $S\subseteq\Zq^n$ and $0\le d\le R$ let
$N_d^S(x)=\#\{u\in S: d(x,u)=d\}$.

\begin{proposition}\label{prop:dct}
Let $A_0$ be the indicator array of $S$ on $\Zq^n$ and $A_1=\dots=A_R=0$.
For each axis $j=1,\dots,n$ and every one-dimensional fiber
$\{x+t e_j\}_{t\in\mathbb{Z}_q}$, perform, for $d=R,R-1,\dots,1$ in this
order,
\[
A_d[a] \;\leftarrow\; A_d[a] + \Bigl(\textstyle\sum_b A_{d-1}[b]\Bigr) -
A_{d-1}[a],
\]
where $a,b$ range over the fiber and the sum is formed before any write in
that fiber and grade. After all $n$ axes, $A_d = N_d^S$ for $0\le d\le R$.
The number of additions is $\Theta(nRq^n)$.
\end{proposition}

\begin{proof}
By induction on the number of processed axes, after axes $1,\dots,j$ the
entry $A_d(x)$ counts the $u\in S$ that differ from $x$ in exactly $d$ of
the first $j$ coordinates and agree with $x$ in the remaining ones: the
update adds, at each position $a$ of the fiber, the number of such $u$ at
grade $d-1$ over the first $j-1$ axes whose $j$-th coordinate differs from
$a$. The descending order of $d$ ensures that the update of $A_d$ reads
$A_{d-1}$ as left by the previous axis, since $A_{d-1}$ is overwritten only
after every $A_d$ that reads it; $A_0$ is never modified. Forming the
fiber sum before writing ensures that no position reads a partially
updated neighbour. Each axis touches every entry of $A_1,\dots,A_R$ a
bounded number of times, which gives the operation count.
\end{proof}

The transform is the fast product-domain technique of Yates~\cite{Yates}
applied to the filtration of $\Zq^n$ by Hamming distance and truncated at
grade $R$. Its inner loop is branchless and identical on all $q^{n-1}$
fibers, and it is bandwidth-bound on a GPU: on the machine of
Appendix~\ref{app:impl} one transform takes $63$\,ms at $q^n=7^{10}$ and
$143$\,ms at $8^{10}$ with $32$-bit counters ($4(R+1)+2$ bytes per word of
$\Zq^n$).

Two instances of Proposition~\ref{prop:dct} are used. With $S$ the set of
uncovered words, $g(x)=\sum_{d\le R}N_d^S(x)$ is the number of words that
a codeword placed at $x$ would newly cover, for every $x\in\Zq^n$; with $S$
the set of words with $\mathrm{cnt}(w)=1$, $\ell(c)=\sum_{d\le R}N_d^S(c)$
is the number of words that deleting $c$ would uncover, for every codeword
$c$.

\subsection{The search}

The search is a large neighbourhood search~\cite{Shaw1998} over these two
maps. Starting from a code (a direct sum or a recorded code), it deletes a
set of codewords chosen among those with small $\ell$, together with some
codewords chosen at random, and then rebuilds by repeatedly placing a
codeword at a maximizer of $g$, recomputing or locally correcting $g$
after each placement. When the rebuilt code has covering radius at most
$R$ at size $M$, the same step is applied at $M-1$. Each placement is
optimal with respect to the exact coverage counts, whereas a move of
Section~\ref{sec:ls} changes one coordinate; the search performs far fewer
moves. The greedy rebuild also serves as a constructor from the empty code:
on $\Kq{8}{10}{4}$ ($q^n=1.07\cdot 10^9$) it produces a code of size
$14946$ and covering radius $4$ in three minutes, while the search of
Section~\ref{sec:ls} did not produce any covering code of that cell.

Five bounds of Table~\ref{tab:ub}, on cells with $2.8\cdot 10^8\le q^n\le
3.5\cdot 10^9$, were found by this search with all arrays in GPU memory,
within hours per cell. The remaining two are described next.

\subsection{Cells with $q^n=10^{10}$}\label{sec:lns:large}

With all arrays in the $96$\,GB of GPU memory the largest cells that can
be searched have $q^n\approx 3.5\cdot 10^9$. For larger cells the arrays
$A_1,\dots,A_R$, which are read and written only by the axis passes of
Proposition~\ref{prop:dct}, are placed in the $480$\,GB of CPU memory,
which the GPU addresses directly, while the array $\mathrm{cnt}$, which is
accessed at random by ball walks, stays in GPU memory. A transform then
takes about $24$ times longer (Table~\ref{tab:hc}). In this configuration
the search found $\Kq{10}{10}{5}\le 5799$, previously $7106$ (a direct
sum), in six hours on one machine, and $\Kq{10}{10}{6}\le 802$, previously
$826$.

\subsection{Cells on which the search fails}\label{sec:lns:fail}

The tabulated bound $\Kq{8}{10}{4}\le 11776$ is the direct sum of a
$\Kq{8}{4}{2}$ code of size $23$ and a $\Kq{8}{6}{2}$ code of size $512$.
In this code every codeword is the sole coverer of many more words than a
codeword of a random code of the same size, so deleting $k$ codewords
uncovers a correspondingly large set that a rebuild with $k-1$ codewords
must cover again. Runs of several hours found no code of size $11775$, and
likewise none below the direct sum $\Kq{9}{10}{4}\le 27\cdot 729$. A code
of size $22$ in $\Kq{8}{4}{2}$ would give $\Kq{8}{10}{4}\le 11264$; local
search, SAT solving (with and without symmetry breaking) and integer
programming neither found such a code nor proved that none exists, and
$\Kq{8}{4}{2}\in\{22,23\}$ remains open.

\section{Lower bounds}\label{sec:lb}

\subsection{The semidefinite program}

Gijswijt and Polak~\cite[Theorem~4.18]{GijswijtPolak} bound $K_q(n,R)^3$
from below by the optimum of a semidefinite program in variables indexed
by the orbits of triples of words under $\mathrm{Aut}(H(n,q))=S_q\wr S_n$,
with positive semidefiniteness constraints from the block diagonalization
of the Terwilliger algebra of the Hamming scheme, linear constraints
expressing sphere-covering inequalities, and the matrix-cut constraints of
their Proposition~4.17. The program is stated in Appendix~\ref{app:sdp}.
For $q\ge 3$ its size depends only on $n$ (at $n=11$: $339$ variables and
$126$ blocks of order at most $13$), and $q$ enters only through the
coefficients; the computations published in~\cite{GijswijtPolak} cover
$q\le 5$.

We reimplemented the reduced program with the generation of all
coefficients in exact integer arithmetic. The implementation was checked
as follows: (a) for each length $2\le n\le 10$ an explicit verified code of
covering radius $R$ (for $6\le n\le 10$, codes from Table~\ref{tab:ub})
was substituted into the program and every constraint was checked in exact
rational arithmetic, with objective value exactly $|C|^3$; (b) every
coefficient generated for six cells with $3\le q\le 7$ was compared with
the output of the authors' published code, with no difference; (c) of
$49$ values of the tables of~\cite{GijswijtPolak} recomputed, $44$ agree to
the displayed precision, and $18$ of the $22$ improved lower bounds
attempted are reproduced exactly; the four others ($\Kq{4}{7}{1}$ by $1$,
$\Kq{5}{7}{1}$ by $7$, $\Kq{5}{8}{1}$ and $\Kq{5}{8}{2}$) were computed in
double precision, and $\Kq{5}{8}{2}\ge 861$ was subsequently reproduced by
the multiprecision computation of Section~\ref{sec:lb:numerics}; (d) no
certified bound exceeds the known upper bound of its cell, checked on all
$1145$ cells of~\cite{Keri}.

\subsection{Certificates}\label{sec:lb:cert}

A dual feasible point of the program (multipliers $y_k\ge 0$ for the linear
constraints and positive semidefinite matrices $Y_b$ for the blocks) proves
the bound it attains (Lemma~\ref{lem:cert}). Each numerical dual solution
is rounded to rationals with a dyadic denominator, each $Y_b$ is displaced
into the interior of the semidefinite cone, and if a dual inequality then
fails in exact arithmetic the whole dual is multiplied by the largest
dyadic $\theta\le 1$ that restores feasibility, which multiplies the bound
by $\theta$.

The certificate is checked by a standalone program (Python standard
library only) that rebuilds the program from $(q,n,R)$ in exact integer
arithmetic, verifies $y_k\ge 0$, verifies $Y_b\succeq 0$ by an exact
$LDL^{\mathsf T}$ factorization over $\mathbb{Q}$, verifies every dual
inequality in integer arithmetic, and derives the integer bound by exact
comparison of cubes. The checker also verifies the validity of the
covering inequality named in the certificate. The checker and the solving
pipeline share the transcription of the program
of~\cite{GijswijtPolak} into coefficient generators; an error in that
transcription would not be detected by a certificate, which is why the
transcription is checked separately by the four tests above. No
floating-point computation enters the checker.

\subsection{Numerical solution}\label{sec:lb:numerics}

The optimum $|C|^3$ can exceed $10^{13}$ while the constraint constants are
$O(1)$ and the objective coefficients span ten orders of magnitude. All
computations therefore use a scaling by powers of two: the objective is
divided by the cube of the sphere-covering bound and each variable is
scaled by the square root of its objective coefficient, rounded to a power
of two, so that the inverse map to the unscaled dual is exact. On
$\Kq{6}{8}{3}$ in double precision this scaling raises the certified bound
from $183$ to $240$; the cube root of the optimum is about $239.52$.

Double-precision interior-point solvers fail on this program beyond optima
of about $1.5\cdot 10^9$ (cube root about $1150$). All $58$ certificates
of Table~\ref{tab:lb} were therefore obtained with SDPA-GMP~\cite{sdpagmp}
at $200$ or $250$ bits of working precision. The problem files carry exact
dyadic coefficients, the multiprecision dual is parsed into exact
rationals, and rounding is performed in the scaled coordinates with
denominator $2^{128}$ (one certificate: $2^{64}$), with the displacement
into the cone proportional to the rounding step. Of the $58$ certificates,
$52$ have $\theta=1$ and the other six have $\theta\ge 1-2.3\cdot
10^{-4}$; in no case does the scaling change the integer bound. The
pipeline reproduces the $512$-bit value $\Kq{5}{8}{2}\ge 861$
of~\cite{GijswijtPolak} exactly.

\subsection{Cells that are not improved}\label{sec:lb:limits}

After the cube root, the semidefinite bound exceeds the sphere-covering
bound by a factor that decreases with $q$ (for $(n,R)=(8,4)$: $1.44$ at
$q=4$ and $1.16$ at $q=7$). It therefore improves the tabulated bounds on
cells whose previous bound is close to the sphere-covering bound, which
are the cells with keys \texttt{y}, \texttt{x}, \texttt{s} and \texttt{a},
and the whole family $(n,R)=(8,2)$, $6\le q\le 21$
(Figure~\ref{fig:summary}(b)). On the cells attempted whose previous bound
is due to~\cite{HHSP2009} (key \texttt{m}, between $11\%$ and $67\%$ above
the sphere-covering bound), the certified value is at most the tabulated
one. For $R=1$ and $q\ge 6$ the bound is close to the sphere-covering
bound: for $\Kq{6}{7}{1}$ the certified value equals $6^7/36=7776$, and
for $\Kq{6}{8}{1}$ it exceeds the sphere-covering bound by $1.5\%$ and is
below the tabulated $41991$.

A certificate whose value does not exceed the tabulated bound is not an
improvement, and a certificate whose value is below the sphere-covering
bound, which occurs when the solver does not converge, carries no
information. Of the $196$ certificate files produced in this work, $58$
improve the tabulated bound, $28$ attain it exactly, $82$ certify a value
between the sphere-covering bound and the tabulated bound, and $28$
certify a value below the sphere-covering bound. Only the $58$ improving
certificates appear in Table~\ref{tab:lb} and in the ancillary files.

The fractional covering linear program does not improve any cell: every
word of $\Zq^n$ lies in exactly $|B_R|$ balls, so summing all covering
constraints gives LP optimum at least $q^n/|B_R|$, and the uniform solution
$1/|B_R|$ attains it.

\section{Verification}\label{sec:verify}

\emph{Upper bounds.} Each code, read from its published file, was checked
by four methods, two of which share no code with the other two: (i)
marking the balls of all codewords in a $q^n$-bit array; (ii)
meet-in-the-middle covering with bitsets over a split of the coordinates;
(iii) computing the minimum distance from every word of $\Zq^n$ to the
code, which also gives the exact covering radius; (iv) $R$-fold dilation
of the indicator array of the code along the coordinate axes, which
computes no distances. A fifth dilation check, written separately from the
four and included in the repository's build script, re-verified all $26$
codes from the published files before this version was prepared. The
number of distinct codewords in each file is compared with the claimed
$M$. For the largest cells the check is itself a computation: the dilation
of $\Kq{9}{10}{5}$ takes a few GB of memory and minutes, that of
$\Kq{10}{10}{5}$ tens of GB or the meet-in-the-middle method; the
ancillary README lists the resource needs.

\emph{Lower bounds.} All $196$ certificates were re-checked by the exact
checker before this version was prepared, with the outcome stated in
Section~\ref{sec:lb:limits}; none certifies a value exceeding a known
upper bound.

The ancillary files contain the $26$ codes (one codeword per line, digits
\texttt{0}--\texttt{9} then \texttt{a}--\texttt{z} for $q>10$), the $58$
certificates of Table~\ref{tab:lb}, and the checkers
\texttt{verify\_cov.py}, \texttt{verify\_independent.py} and
\texttt{certify.py}.

\subsection*{Data availability}
The codes, the certificates, the checkers, and a README are ancillary
files of the arXiv version of this paper. The search, certification, and
verification source code, together with a machine-readable table merging
the bounds of~\cite{Keri} with the post-2011 literature and the present
improvements, is available at \url{https://github.com/Mapika/coldcase}.

\subsection*{Use of artificial intelligence}
The manuscript was drafted and edited, and the search, certification, and
verification software written, by the AI system Claude (Anthropic) under
the direction of the author. The author has verified the statements and
proofs in this paper and takes full responsibility for its content.

\bibliographystyle{amsplain}
\bibliography{refs}

\appendix

\section{Implementation measurements}\label{app:impl}

All measurements were made on one machine with $64$ ARM CPU cores,
$480$\,GB of CPU memory and one GPU with $96$\,GB of memory, the CPU
memory being addressable by the GPU. Comparative figures are process CPU
time, or paired measurements in the same time window on the same machine.

\subsection{The evaluation of Section~\textup{\ref{sec:ls:eval}}}
\label{app:impl:cpu}

On three reference cells of increasing size ($\Kq{6}{6}{3}$ at $M=41$,
$\Kq{6}{8}{4}$ at $M=169$, $\Kq{8}{9}{4}$ at $M=2944$) candidate
evaluation takes $92.9\%$, $98.5\%$ and $99.6\%$ of the run time, and one
evaluated sphere pattern costs from $1.9$\,ns to $7.4$\,ns as the counter
array grows from $93$\,kB to $268$\,MB. Table~\ref{tab:engfixed} gives the
speedup at a fixed search trajectory and Table~\ref{tab:engttt} the
speedup to a fixed quality target.

\begin{table}[ht]
\centering
\begin{tabular}{lccc}
\toprule
variant & $\Kq{6}{6}{3}$ & $\Kq{6}{8}{4}$ & $\Kq{8}{9}{4}$ \\
        & 20\,000 it & 600 it & 50 it \\
\midrule
reference implementation      & $1.00\times$ & $1.00\times$ & $1.00\times$ \\
unrolled walk, precomputed offsets  & $1.35\times$ & $1.29\times$ & $1.13\times$ \\
\quad + 8-bit state array     & $1.28\times$ & $1.65\times$ & $1.65\times$ \\
\quad + 2-bit state array     & $0.82\times$ & $1.34\times$ & $2.44\times$ \\
shared sphere, 8-bit state    & $1.32\times$ & $1.95\times$ & $1.95\times$ \\
\quad + uncovered-word list   & $2.21\times$ & $5.57\times$ & $6.06\times$ \\
\quad with 2-bit state        & $2.14\times$ & $\mathbf{6.32\times}$ & $\mathbf{9.86\times}$ \\
\bottomrule
\end{tabular}
\caption{Speedup in CPU time at a fixed iteration budget, median of $9$
runs ($6$ on the largest cell). All rows perform the same search: the
final uncovered counts agree, seed by seed, across rows. The interquartile
ranges are $0.9\%$, $2.2\%$ and $5.6\%$ of the median.}
\label{tab:engfixed}
\end{table}

\begin{table}[ht]
\centering
\begin{tabular}{lcccc}
\toprule
cell & $M$ & target & identical search & all $n(q-1)$ moves \\
\midrule
$\Kq{6}{6}{3}$ & 41   & $|U|\le 30$      & $2.03\times$ $[1.99,2.13]$ & $0.28\times$ \\
$\Kq{6}{8}{4}$ & 169  & $|U|\le 20$      & $6.12\times$ $[6.06,6.62]$ & $8.03\times$ $[5.75,11.02]$ \\
$\Kq{8}{9}{4}$ & 2944 & $|U|\le 10\,800$ & $8.07\times$ $[6.19,10.72]$ & $11.04\times$ $[8.06,12.17]$ \\
\bottomrule
\end{tabular}
\caption{Median per-seed speedup in CPU time to the target over $20$
seeds, with percentile-bootstrap $95\%$ confidence intervals. In the
fourth column the optimized program reached the target at the same
iteration as the reference on every seed. The fifth column also evaluates
all $n(q-1)$ moves of each candidate codeword, which changes the search.}
\label{tab:engttt}
\end{table}

The following modifications were tested and not adopted. Caching move
evaluations between iterations requires the evaluation spheres of
successive moves to be disjoint, that is $d(c,c')>2R+2$; since $2R+2\ge n$
on every cell considered, every move invalidates every cached value.
Maintaining the list of words with $\mathrm{cnt}(w)=1$ in the same way as
$U$ pays only if $n|U_1|<\binom{n}{R}(q-1)^R$, which fails by one to two
orders of magnitude. Fixing one codeword at $0^n$ gave no measurable
change. Choosing as focus the word uncovered for the longest time instead
of a random uncovered word performed much worse. Restarts reduced the time
to a hard target on $\Kq{6}{6}{3}$ (speedup $4.2\times$, $p=0.002$) and
had no effect on the cells of Table~\ref{tab:ub} (coefficient of variation
of the time to target $0.20$--$0.47$; $p=1.0$).

\subsection{Choice of the configuration of Section~\textup{\ref{sec:ls:config}}}
\label{app:impl:arena}

The four implementations were compared under one protocol: identical
interface and resource limits, a set of development instances, a set of
held-out evaluation instances, $15$ seeds per instance, and a score
computed from verified output only (a verified covering code scores
$1000$, a partial cover scores minus its uncovered fraction, and invalid
output scores below every valid partial cover). The restricted searches
scored highest on the held-out instances, the search of
Section~\ref{sec:ls:eval} was fastest on large cells, the SIMD variant on
cache-resident cells, and single-threaded processes scored higher than
multi-threaded ones except on the largest cells. Under the same protocol
the combined program solved $14$ of $20$ held-out instances, against $12$
of $20$ for the best single implementation.

\subsection{GPU measurements for Section~\textup{\ref{sec:lns}}}
\label{app:impl:gpu}

\emph{Ports of the local search.} A formulation of candidate evaluation as
integer matrix products has contracted dimension $nq\approx 50$ and did
not reach CPU throughput. A port running $256$ independent search
processes in one persistent kernel reached $0.09$, $0.10$ and $1.24$ times
the throughput of the $64$ CPU cores on cells with $q^n=4.7\cdot 10^4$,
$1.7\cdot 10^6$ and $1.3\cdot 10^8$.

\emph{Memory placement for $q^n>3.5\cdot 10^9$
(Section~\textup{\ref{sec:lns:large}}).} The GPU addresses the CPU memory
through a cache-coherent interconnect with native atomic operations. The
arrays $A_1,\dots,A_R$ are placed in CPU memory and $\mathrm{cnt}$ in GPU
memory (for $\Kq{10}{10}{5}$: $240$\,GB and $21$\,GB). The placement must
be fixed explicitly: allocation through the unified-memory allocator is
about $30$ times slower than \texttt{malloc} with direct access through
the coherent page tables, and with plain \texttt{malloc} the operating
system migrates pages touched by the GPU into GPU memory until it is full;
binding the allocations to the CPU memory node with \texttt{mbind} makes
the placement deterministic. The outputs were identical across all
placements. Table~\ref{tab:hc} gives the time of one transform.

\begin{table}[ht]
\centering
\begin{tabular}{lrrrr}
\toprule
cell & $q^n$ & arrays & placement & transform \\
\midrule
$\Kq{8}{10}{4}$ & $1.1\cdot 10^9$ & $24$\,GB & GPU memory & $0.143$\,s \\
$\Kq{8}{10}{4}$ & $1.1\cdot 10^9$ & $24$\,GB & CPU memory (bound) & $3.4$\,s \\
$\Kq{8}{10}{4}$ & $1.1\cdot 10^9$ & $24$\,GB & CPU memory (managed) & $8.7$\,s \\
$\Kq{9}{10}{5}$ & $3.5\cdot 10^9$ & $91$\,GB & CPU memory (bound) & $22$--$26$\,s \\
$\Kq{10}{10}{5}$ & $10^{10}$ & $260$\,GB & CPU memory (bound) & $73$--$110$\,s \\
\bottomrule
\end{tabular}
\caption{Time of one transform of Proposition~\ref{prop:dct} by placement
of $A_1,\dots,A_R$, on an idle machine; $\mathrm{cnt}$ is in GPU memory in
all rows.}
\label{tab:hc}
\end{table}

\emph{Ball walks.} Ball walks use a precomputed table of packed
coordinate-shift patterns, so that each visited word costs one table load
and at most $R$ additions in shared memory, without division or modulo
operations. Splitting each walk over up to $8192$ thread blocks (instead
of one) makes single-word updates $30$ times faster on $\Kq{8}{9}{4}$
($0.80\to 0.027$\,ms) and $37$ times faster on $\Kq{8}{10}{4}$ ($1.34\to
0.036$\,ms), batched candidate evaluations $3.7$--$4.3$ times faster, and
a complete greedy construction on $\Kq{8}{9}{4}$ $1.4$ times faster at
identical output. Recording each word's last coverer during the marking
pass, so that all $\ell(c)$ are obtained from one linear scan, replaces
one transform but is slower ($28$ vs.\ $17$\,ms on $\Kq{8}{9}{4}$, $669$
vs.\ $145$\,ms on $\Kq{8}{10}{4}$), since the marking pass performs
$M\cdot|B_R|$ atomic operations while the transform streams $O(nRq^n)$
words.

\subsection{Baselines}\label{app:impl:baseline}

\emph{Integer programming.} The set-covering integer program with one
binary variable and one constraint per word of $\Zq^n$ has $q^n\cdot|B_R|$
nonzero coefficients. Among the improved cells with $q^n\le 10^8$ only
$\Kq{6}{7}{3}$ ($1.4\cdot 10^9$ nonzeros, about $17$\,GB) fits in
$480$\,GB of memory; the other $n=6$ cells need between $137$\,GB and
$7.2$\,TB for the constraint matrix. On $\Kq{6}{7}{3}$, HiGHS with $8$
threads and one hour per run produced neither a feasible cover nor a dual
bound above the sphere-covering bound, in minimization mode and in
feasibility mode at sizes $245$ and $227$.

\emph{Plain tabu search.} A tabu search from the same code base without
the modifications of Section~\ref{sec:ls:eval} was run against the program
of Section~\ref{sec:ls:config} with $8$ CPU-hours per cell, identical
seeds and the same time window:

\begin{center}
\begin{tabular}{lccc}
\toprule
cell (target $M$) & \cite{Keri} & plain tabu & Section~\ref{sec:ls:config} \\
\midrule
$\Kq{8}{6}{4}$ (19) & 20 & solved & solved \\
$\Kq{12}{6}{4}$ (40, 39) & 41 & solved, solved & solved, solved \\
$\Kq{13}{6}{4}$ (45) & 46 & solved & solved \\
$\Kq{14}{6}{4}$ (51, 50) & 52 & solved, solved & solved, solved \\
$\Kq{15}{6}{4}$ (58, 57) & 59 & solved, failed & solved, solved \\
$\Kq{6}{7}{3}$ (245, 227) & 246 & solved, failed & solved, failed \\
\bottomrule
\end{tabular}
\end{center}

The $n=6$ improvements other than $\Kq{15}{6}{4}\le 57$ are thus
reachable by a plain tabu search with this budget. Neither program reached
$\Kq{6}{7}{3}\le 227$ from scratch with this budget; that code was
obtained by the sequence of searches at decreasing $M$ described in
Section~\ref{sec:ls:search}.

\section{The certified program}\label{app:sdp}

The reduced program of~\cite[Theorem~4.18]{GijswijtPolak} (non-binary
case, $q\ge 3$) is a minimization in variables $x\in\mathbb{R}^N$ indexed
by the orbits of triples of words under $\mathrm{Aut}(H(n,q))=S_q\wr S_n$:
\begin{equation}\label{eq:sdp}
\begin{aligned}
\min\; c^{\mathsf T}x
\quad\text{s.t.}\quad
& x\ge 0,\qquad
\langle l_k,x\rangle + l_k^0 \ge 0 \quad (k=1,\dots,m),\\
& C_b + \textstyle\sum_{v=1}^{N} x_v A_b^v \succeq 0 \quad (b=1,\dots,B),
\end{aligned}
\end{equation}
where the data $c$, $(l_k,l_k^0)$, $(C_b,A_b^v)$ are rationals determined
by $(q,n,R)$: the objective and the linear constraints encode
sphere-covering inequalities and the matrix-cut inequalities
of~\cite[Proposition~4.17]{GijswijtPolak}, and the semidefinite blocks
come from the block diagonalization of the Terwilliger algebra of the
nonbinary Hamming scheme~\cite[\S 3]{GijswijtPolak}. By
\cite[Theorem~4.18]{GijswijtPolak}, every code $C\subseteq\Zq^n$ of
covering radius at most $R$ gives a feasible point $x_C$
of~\eqref{eq:sdp} with $c^{\mathsf T}x_C=|C|^3$; hence
$K_q(n,R)^3\ge\mathrm{OPT}$, the optimum of~\eqref{eq:sdp}.

A certificate for $(q,n,R)$ consists of $y\in\mathbb{Q}^m_{\ge 0}$ and
rational symmetric positive semidefinite matrices $Y_1,\dots,Y_B$, stored
over one common denominator. Let
\[
d_v \;=\; \sum_k y_k\, l_k[v] \;+\; \sum_b \langle Y_b, A_b^v\rangle
\quad (v=1,\dots,N),
\qquad
d_0 \;=\; \sum_k y_k\, l_k^0 \;+\; \sum_b \langle Y_b, C_b\rangle .
\]

\begin{lemma}\label{lem:cert}
If $d_v\le c_v$ for every $v$, then every feasible point $x$
of~\eqref{eq:sdp} satisfies $c^{\mathsf T}x\ge -d_0$; hence
$K_q(n,R)\ge\bigl\lceil(-d_0)^{1/3}\bigr\rceil$.
\end{lemma}

\begin{proof}
For feasible $x$,
\[
c^{\mathsf T}x + d_0
= \sum_v (c_v-d_v)\,x_v
+ \sum_k y_k\bigl(\langle l_k,x\rangle + l_k^0\bigr)
+ \sum_b \Bigl\langle Y_b,\; C_b+\sum_v x_v A_b^v\Bigr\rangle
\;\ge\; 0,
\]
since each summand is a product of nonnegative quantities ($x\ge 0$ and
$c_v-d_v\ge 0$; $y_k\ge 0$ and the linear constraint value $\ge 0$; the
inner product of two positive semidefinite matrices is nonnegative). The
bound on $K_q(n,R)$ follows from $K_q(n,R)^3\ge\mathrm{OPT}\ge -d_0$.
\end{proof}

If the rounded dual has $d_v>c_v$ for some $v$, the certificate is
multiplied by the largest dyadic $\theta\le 1$ with $\theta d_v\le c_v$ for
all $v$, which gives the bound $-\theta d_0$. The checker verifies the
hypotheses of Lemma~\ref{lem:cert} (nonnegativity of $y$, positive
semidefiniteness of each $Y_b$ by exact $LDL^{\mathsf T}$ factorization
over $\mathbb{Q}$, and the $N$ inequalities $d_v\le c_v$ in integer
arithmetic) and rebuilds the data $c,(l_k,l_k^0),(C_b,A_b^v)$ from
$(q,n,R)$; the transcription of that data from~\cite{GijswijtPolak} is
the part of the computation not covered by the certificate, and it is
checked as described in Section~\ref{sec:lb}.

\section{The code establishing $\Kq{6}{8}{4}\le 166$}\label{app:code}
The code is printed here because it is the largest relative improvement
of Table~\ref{tab:ub} that fits on a page (the codes for
$\Kq{7}{9}{5}\le 240$ and $\Kq{8}{10}{5}\le 1883$ are longer). Each row
lists four codewords of $\mathbb{Z}_6^8$; the machine-readable version is
the ancillary file \texttt{K6\_8\_4\_M166.txt}.

{\scriptsize
\begin{verbatim}
01430421  12541532  23052043  30103154
45214205  50325310  03405225  14510330
25021441  30132552  41243003  52354114
03134001  14245112  25350223  30401334
51512445  52010550  02442243  13553304
24004415  35115520  40220031  51331142
02303231  13414342  24525453  35030504
40141015  51252120  05143544  10254055
21305100  32410211  43521322  54032433
04301535  13432040  20543151  31054202
42105313  53210434  04043240  15154351
20205402  31310513  42421024  53532135
05111133  10222244  21333355  32444400
43555511  54000022  01045452  12150503
23201014  34312125  45423230  50534341
00544125  11055230  22100341  33211452
44322503  55433014  02231500  13342011
24453122  35504233  40015344  51120455
05223315  10334420  21445531  32550042
43001153  54112204  03411050  14542101
25033212  30144323  41255434  52300445
04450305  15501410  20012521  31123032
42234143  53345254  00052413  11103524
22214035  33324140  44430251  55541302
02115012  13220123  24331234  35442345
40553450  51004501  01024314  12135425
23240530  34351041  45402152  50513203
05352332  10413443  21514554  52025505
43130110  54241221  04205542  15315053
20420104  31531215  42042320  53153431
01213111  12324222  23435333  34540444
45051555  42131403  00300252  11411303
22522414  33033525  44104030  55255141
05535024  10040135  21151240  32202351
44313402  54424513  01500043  12051134
23122205  34233310  45344421  50455532
03314404  10402515  25530020  31021031
41132242  52243353  12005001  03402520
31455113  23344502  55203541  32515205
30325000  02023150  14021545  54142054
34124521  03523252  45250342  00522424
44553234  53044110
\end{verbatim}

}

\end{document}